\documentclass[11pt]{amsart}

\usepackage{amsmath}%
\usepackage{amsfonts}%
\usepackage{amssymb}%
\usepackage{amsthm}
\usepackage{todonotes}
\usepackage{comment}

\usepackage[lite, alphabetic, nobysame]{amsrefs}

\usepackage{url}
\usepackage{geometry}
\usepackage{enumerate}
\theoremstyle{plain}
\newtheorem{theorem}{Theorem}

\newtheorem{lemma}[theorem]{Lemma}

\newtheorem{claim}[theorem]{Claim}

\theoremstyle{definition}
\newtheorem{definition}[theorem]{Definition}
\newtheorem{example}[theorem]{Example}

\newcommand{\M}{\mathcal{M}}
\newcommand{\Ell}{\mathcal{L}}
\newcommand{\N}{\mathbb{N}}
\title{Quasi--invariant measures concentrating on countable structures}

\author{Clinton Conley}
\address{Department of Mathematical Sciences, Carnegie Mellon University, 5000 Forbes Ave., Pittsburgh, PA 15213, USA}
\email{clintonc@andrew.cmu.edu}

\author{Colin Jahel}
\address{TU Dresden, 01099 Dresden, Germany}
\email{colin.jahel@tu-dresden.de}

\author{Aristotelis Panagiotopoulos}
\address{Kurt G\"odel Research Center, Faculty of Mathematics,   Universit\"at Wien, Kolingasse 14-16, 1090 Vienna, Austria}
\email{aristotelis.panagiotopoulos@gmail.com}

\thanks{
This research was supported by the
NSF Grants DMS-2154160 and DMS-2154258. C. Jahel was partially supported by  DFG (Project FinHom,
Grant 467967530).}

\begin{document}

\begin{abstract}
Countable  $\mathcal{L}$-structures $\mathcal{N}$  whose isomorphism class supports a permutation invariant probability measure in the logic action have been characterized by  Ackerman--Freer--Patel to be precisely those  $\mathcal{N}$  which have no algebraicity. Here we characterize those countable $\mathcal{L}$-structure  $\mathcal{N}$  whose isomorphism class supports a  quasi--invariant probability measure. These turn out to be precisely those  $\mathcal{N}$ which are not ``highly algebraic" ---we say that  $\mathcal{N}$ is highly algebraic if outside of every finite $F$ there is some $b$ and  a tuple  $\bar{a}$ disjoint from $b$ so that $b$ has a finite orbit under the pointwise stabilizer of $\bar{a}$ in $\mathrm{Aut}(\mathcal{N})$.
As a bi-product of our proof we show that whenever the isomorphism class of $\mathcal{N}$ admits a  quasi-invariant measure, then it admits one with continuous   Radon--Nikodym cocycles. 
\end{abstract}

\maketitle

\section{Introduction}

The past decade has seen the development of the study of invariant random structures. This field was kindled by the work of Petrov and Vershik \cite{PV}   on random triangle-free graphs. A fundamental result in this area is due to Ackerman, Freer and Patel \cite{AFP} who prove that, for a given countable $\mathcal{L}$-structure $\M$,  the existence of an invariant random structure  almost surely isomorphic to $\M$ is equivalent to   $\M$  satisfying  a combinatorial property known as no-algebraicity. 

Recall that for any first order language $\mathcal{L}$ which  contains countably many  relation and function symbols, the collection $\mathrm{Str}(\Ell,\N)$ of all $\mathcal{L}$-structures whose domain is the natural numbers $\mathbb{N}$ can be endowed with a natural Polish topology; see Section \ref{SS:Background:1}. 
For every structure $\mathcal{N}\in \mathrm{Str}(\mathcal{L},\mathbb{N})$, the isomorphism class $[\mathcal{N}]_{\mathrm{iso}}$ of $\mathcal{N}$ is the set:
\[[\mathcal{N}]_{\mathrm{iso}}:=\{\mathcal{M}\in\mathrm{Str}(\mathcal{L},\mathbb{N}) \colon \mathcal{M} \text{ and } \mathcal{N}  \text{ are isomorphic}\}.\]
These  classes can be identified with the orbits under the \emph{logic action} $\mathrm{Sym}(\mathbb{N})\curvearrowright \mathrm{Str}(\mathcal{L},\mathbb{N})$  of the Polish group $\mathrm{Sym}(\mathbb{N})$  of all permutations of $\mathbb{N}$  on $\mathrm{Str}(\mathcal{L},\mathbb{N})$; see Section \ref{SS:Background:2}.

A {\bf random structure} is just a Borel probability measure on $\mathrm{Str}(\Ell,\N)$. A random structure is said to be {\bf invariant} if it is invariant under the logic action. In \cite{AFP} it was shown that the existence of an invariant random structure supported on $[\mathcal{N}]_{\mathrm{iso}}$ is equivalent to $\mathcal{N}$ having {\bf no-algebraicity}:  for any tuple $\bar a$, $\mathrm{acl}(\bar{a},\mathcal{N})=\bar{a}$ where $\mathrm{acl}(\bar{a},\mathcal{N})$ refers to the  points of $\mathcal{N}$ whose orbit is finite under the pointwise stabilizer of $\bar{a}$; see Section \ref{SS:Background:1}. This result was followed by an extensive study of invariant random structures for instance in \cite{AFKrucmanP}, \cite{AFKwiatowskaP}, \cite{CraneTowsner} and by study of invariant random expansions in \cite{JahelJoseph}.

In the present paper we study a generalization of invariant random structures: quasi-invariant random structures. Let $G\curvearrowright X$ be a Borel action on a Polish group on a standard Borel space and recall that a  probability measure $\mu$ on $X$ is {\bf $G$-quasi--invariant},  if for every Borel $B\subseteq X$ and every $g\in G$ we have that $\mu(B)=0 \iff \mu (g\cdot B)=0$. The class of $S_\infty$-quasi-invariant measures have been studied by Nessonov in \cite{Nessonov}.

We call a countable structure $\mathcal{N}$ {\bf quasi--random} if there exists a $\mathrm{Sym}(\mathbb{N})$-quasi--invariant probability measure $\mu$ on $\mathrm{Str}(\mathcal{L},\mathbb{N})$ so that $\mu([\mathcal{N}]_{\mathrm{iso}})=1$.  
\begin{definition}
We say that $\mathcal{N}$ is {\bf highly algebraic} if for every finite $F\subseteq \mathrm{dom}(\mathcal{N})$ there exists a  tuple $\bar{a}=(a_0,\ldots,a_{n-1})$ in $\mathrm{dom}(\mathcal{N})$ so that $\mathrm{acl}(\bar{a},\mathcal{N})\not\subseteq F\cup\{a_0,\ldots,a_{n-1}\}$. 
\end{definition}

The main result of this paper is the following characterization of quasi--random structures in the spirit of  \cite{AFP}.

\begin{theorem}\label{T:main}
A countable structure is quasi--random if and only if it is not highly algebraic.
\end{theorem}

It is to be noted that since, by a result of Becker and Kechris (see \cite{BK96} section 2.7), when $\mathcal{L}$ contains symbols of arbitrarily large arity, the action of $\mathrm{Sym}(\N)$ on $\mathrm{Str}(\Ell,\N)$ is universal, this theorem describes all (up to embedding) $\mathrm{Sym}(\N)$-quasi--invariant actions that admit an orbit of full measure.

Recall that to every quasi-invariant measure $\mu$ on a Borel $G$-space $X$ one can use the Radon-Nikodym derivatives to associate to $\mu$ a measurable cocycle $c_{\mu}\colon G\times X \to \mathbb{R}$.
The following comes as a bi-product of our proof.

\begin{theorem}\label{T:main2}
If $\mathcal{N}\in \mathrm{Str}(\mathcal{L},\mathbb{N})$ is a not highly algebraic structure, then $[\mathcal{N}]_{\mathrm{iso}}$  is the support of a $\mathrm{Sym}(\mathbb{N})$-quasi--invariant probability measure $\mu$  whose  cocycle $c_{\mu}$ induces a continuous map $c_{\mu}(-,\mathcal{N})\colon \mathrm{Sym}(\mathbb{N}) \to \mathbb{R}$ for each $\mathcal{N}\in  \mathrm{Str}(\mathcal{L},\mathbb{N})$.
\end{theorem}

\subsection*{Outline and sketch of proofs} Section 2 covers the essential definitions and results we need in the proofs. Section 3 is concerned with showing that highly algebraic structures are not quasi-random; the proof consists of constructing continuum many disjoint translates of any compact subset of the orbit of the structure. Section 4 shows the forward direction of Theorem \ref{T:main}. Here we 
uses invariant measures from \cite{AFP} to construct quasi-invariant measures concentrated on the orbit of a given non highly algebraic structure. 
Section 5 collects some useful examples of $\mathcal{L}$-structures in relation to high algebraicity.

\section{Background}
\subsection{Countable $\mathcal{L}$-structures}\label{SS:Background:1}
We follow the usual conventions from model theory; see e.g. \cite{Hodges} for more details.
A countable language $\mathcal{L}$ is a collection of  function symbols and  and relation symbols. Each symbol has a pre-specified arity. An $\Ell$-structure $\mathcal{M}$ is a set $M$ together with an interpretation of $\Ell$, i.e. for each relation symbol $R\in \Ell$ of arity $n(R)\in \mathbb{{N}}$ a subset $R^\mathcal{M}$ of $M^n$, for each  function symbol $f\in \Ell$ of arity $n(f)\in\mathbb{N}$, a map $f^\mathcal{M}\colon M^n \to M$. Given  such $\mathcal{L}$ and  some countable set  $N$, let $\mathrm{Str}(\mathcal{L},N)$ be the space of all $\mathcal{L}$-structures $\mathcal{N}$ with $\mathrm{dom}(\mathcal{N})=N$. As usual, we identify $\mathrm{Str}(\mathcal{L},N)$ with the Polish space
\[\big(\prod_{R \in \mathcal{L}} 2^{\overbrace{N\times\cdots \times N}^{n(R)}} \big)\times \big(\prod_{f \in \mathcal{L}} N^{{\overbrace{N\times\cdots \times N}^{n(f)}}} \big), \]
where the  products are taken over all  relation and function symbols in $\mathcal{L}$, respectively.

Whenever the domain $N$ is fixed, we will denote by $a,b,c,\ldots$ the elements of $N$ and by $\bar{a},\bar{b}, \bar{c}\ldots$ tuples in $N$. By a {\bf tuple} we mean any finite map $\bar{a} \colon n\to N$ where  $n=\{0,\ldots,n-1\}$.  Such tuple has length $n$ and we often denote it by $(a_0,\ldots,a_{n-1})$. and we
denote by $\{\bar{a}\}$ the associated set $\{a_0,\ldots,a_{n-1}\}$ of its elements. We denote by  $\bar{a}\bar{b}$ the concatenation $(a_0,\ldots,a_{n-1},b_0,\ldots,b_{m-1})$ of the tuples $\bar{a}$ and $\bar{b}$.

Let  $\mathcal{N}$ be an $\mathcal{L}$-structure with $\mathrm{dom}(\mathcal{N})=N$ and let $\bar{a}$  in $N$. The {\bf quantifier free type} $\mathrm{qft}(\bar{a},\mathcal{N})$ of $\bar{a}$ in $\mathcal{N}$ is the collection of all quantifier free $\mathcal{L}$-formulas $\varphi(\bar{x})$ with $\mathcal{N}\models \varphi(\bar{a})$. 
We say that $\mathcal{N}$ is {\bf ultrahomogeneous} if $\mathrm{qft}(\bar{a},\mathcal{N})= \mathrm{qft}(\bar{b},\mathcal{N})$ holds if and only if there exists some $g\in \mathrm{Aut}(\mathcal{N})$ with $g\bar{a}=\bar{b}$.
A relation $R\in \mathcal{L}$ {\bf captures $\mathrm{qft}(\bar{a},\mathcal{N})$}  if  for every $\bar{b}$ in $N$ we have that $\mathcal{N}\models R(\bar{b})$ if and only if $\mathrm{qft}(\bar{a},\mathcal{N})= \mathrm{qft}(\bar{b},\mathcal{N})$.

Finally,  we say that  $\mathcal{N}$ is {\bf R--ultrahomogeneous} if it is ultrahomogeneous and for every $\bar{a}$ in $N$ we have that $\mathrm{qft}(\bar{a},\mathcal{N})$ is captured by some relation $R$ in $\mathcal{L}$. Let
 \[\mathrm{Aut}(\mathcal{N})_{\bar{a}}:=\{g\in \mathrm{Aut}(\mathcal{N}) \colon g(a_0)=a_0, \ldots, g(a_{n-1})=a_{n-1} \}\]
 be the pointwise stabilizer of $\bar{a}$. 
The {\bf algebraic closure of $\bar{a}$ in $\mathcal{N}$} is the set
\[\mathrm{acl}(\bar{a},\mathcal{N}):=\{ b \in \mathrm{dom}(\mathcal{N}) \colon \mathrm{Aut}(\mathcal{N})_{\bar{a}}\cdot b \text{ is finite} \}.\]

The following standard fact is going to be used later on.

\begin{lemma}\label{L:Ultrahomogeneous}
For every countable $\mathcal{L}$-structure $\mathcal{N}$ there exists a countable  $\mathcal{L}'\supseteq \mathcal{L}$ and an  $R$--ultrahomogeneous $\mathcal{L}'$-structure $\mathcal{N}'$ whose $\mathcal{L}$-reduct $\mathcal{N}'\upharpoonright \mathcal{L}$ is $\mathcal{N}$ and for  $\bar{a}$ we have that:
\[\mathrm{acl}(\bar{a},\mathcal{N})=\mathrm{acl}(\bar{a},\mathcal{N}')\quad \text{ and }\quad  \mathrm{Aut}(\mathcal{N})_{\bar{a}}=\mathrm{Aut}(\mathcal{N}')_{\bar{a}}.\]
\end{lemma}
\begin{proof} This follows from the usual ``orbit completion" argument which can be find in \cite{Hodges}*{Theorem 4.1.4} or \cite{BK96}*{Section 1.5}. In short, for every $n\in 1,2,\ldots$, and for every orbit $O\subseteq \mathrm{dom}(\mathcal{N})^n$ in the diagonal action $g\cdot(a_0,\ldots,a_{n-1})=(g\cdot a_0,\ldots,g\cdot a_{n-1})$ of $\mathrm{Aut}(\mathcal{N})$ on $\mathrm{dom}(\mathcal{N})^n$, we introduce a $n$-ary relation symbol $R^{n,O}$ and we set $\mathcal{N}'\models R^{n,O}(\bar{a})$ if and only if $\bar{a}\in O$.  
\end{proof}

\subsection{Logic action}\label{SS:Background:2}
Let $\mathrm{Sym}(N)$ be the Polish group of all bijections $g\colon N\to N$. The group operation is the operation of composition $(g\cdot h)(n)=g(h(n))$ and the topology is that of pointwise convergence. The {\bf logic action} is the  continuous action $\mathrm{Sym}(N)\curvearrowright \mathrm{Str}(\mathcal{L},N)$ given by the {\bf push-forward} operation: 
\begin{gather*}
    (g,\mathcal{M})\mapsto g_* \mathcal{M}, \text{where}\\
    g_*\mathcal{M}\models R(a_0,\ldots,a_{n-1}) \iff \mathcal{M}\models R(g^{-1}(a_0),\ldots,g^{-1}(a_{n-1})), \text{ and}\\
    g_*\mathcal{M}\models \big(b=f(a_0,\ldots,a_{n-1}) \big)\iff \mathcal{M}\models \big(g^{-1}(b) =f(g^{-1}(a_0),\ldots,g^{-1}(a_{n-1}))\big).
\end{gather*}

\section{Highly algebraic structures are not quasi--random}

We fix some  $\mathcal{N}\in\mathrm{Str}(\mathcal{L},\mathbb{N})$ and we assume throughout this section that $\mathcal{N}$ is  highly algebraic and  $R$--ultrahomogeneous. We also set  $X:=[\mathcal{N}]_{\mathrm{iso}}$ to be the isomorphism class of $\mathcal{N}$ in $\mathrm{Str}(\mathcal{L},\mathbb{N})$.  
Notice that the ultrahomogeneity of $\mathcal{N}$ implies that $X$ is a $G_{\delta}$ subset of   $\mathrm{Str}(\mathcal{L},\mathbb{N})$ and hence a Polish space. The main goal of this section is to prove the following:

\begin{lemma} \label{L:aris2}
    If $K\subseteq X$ is compact, then there is a family $(g_{\alpha}\colon \alpha\in 2^{\mathbb{N}})$ of permutations $g_{\alpha}\in \mathrm{Sym}(\mathbb{N})$ so that for all $\alpha,\beta\in 2^{\mathbb{N}}$ with $\alpha\neq \beta$ we have that $g_{\alpha} K \cap g_{\beta} K =\emptyset$.
\end{lemma}

This lemma directly implies one of the directions of Theorem \ref{T:main}:

\begin{proof}[Proof of $(\Rightarrow)$ direction of Theorem \ref{T:main} from Lemma \ref{L:aris2}]
Let  $\mathcal{N}^*\in  \mathrm{Str}(\mathcal{L}^*,\mathbb{N})$ be some highly algebraic structure and assume towards contradiction that there exists a $\mathrm{Sym}(\mathbb{N})$-quasi--invariant probability measure $\mu$ on $\mathrm{Str}(\mathcal{L}^*,\mathbb{N})$ so that $\mu([\mathcal{N}^*]_{\mathrm{iso}})=1$. Let $\mathcal{L}:=(\mathcal{L}^*)'$ and  $\mathcal{N}:=(\mathcal{N}^*)'$ be the language and the structure provided by Lemma \ref{L:Ultrahomogeneous} and set $X:=[\mathcal{N}]_{\mathrm{iso}}$.

Consider now the map $\pi\colon X \to \mathrm{Str}(\mathcal{L},\mathbb{N})$,  which implements the $\mathcal{L}$-reduct $\mathcal{M}\mapsto \mathcal{M}\upharpoonright \mathcal{L}$.  Then $\pi$ is clearly a  $\mathrm{Sym}(\mathbb{N})$--equivariant Borel isomorphism and hence we can transfer  $\mu$ to a  $\mathrm{Sym}(\mathbb{N})$-quasi--invariant probability measure on $X$ by pushing it forward via $\pi^{-1}$. Since  $\mathcal{N}$ is ultrahomogeneous, we have that $X$ is a Polish space, and Lemma \ref{L:aris2} applies. But probability measures on Polish spaces are regular and hence, there exists a compact $K\subseteq X$ so that $\mu(K)>0$. By quasi-invariance of $\mu$ we have that $\mu(g_{\alpha}K)>0$, for all $\alpha\in 2^{\mathbb{N}}$; where $(g_{\alpha}\colon \alpha\in 2^{\mathbb{N}})$ is as in Lemma \ref{L:aris2}. Hence there is some $\varepsilon>0$ so that $\{\alpha \in 2^{\mathbb{N}}\colon \mu(g_\alpha K)>\varepsilon \}$ is uncountable. But since $g_{\alpha} K \cap g_{\beta} K =\emptyset$ when $\alpha\neq \beta$, $\sigma$-additivity of $\mu$ contradicts that $\mu$ is a probability measure.
\end{proof}

The proof of  Lemma \ref{L:aris2} occupies the rest of this section. Interestingly, there are two distinct ways for a structure $\mathcal{N}$ to exhibit high algebraicity. Each  way contributes to the proof of  Lemma \ref{L:aris2} through a  different argument. In Lemma \ref{L:aris1.5} below we establish this dichotomy--in--behaviour, before  we conclude with the proof of Lemma \ref{L:aris2}.

\subsection{Highly algebraic structures}

Below we will make extensive use of the Neumann's lemma \cite{neumann1976structure}*{Lemma 2.3}. Rephrased  as in \cite{Cam}*{Theorem 6.2}, it  states:

\begin{lemma}[Neumann] \label{Thm:Neu} Let $H$ be a group acting on a countable set $\Omega$ with no finite orbit. Let $\Gamma$ and $\Delta$ be finite subsets of $\Omega$, then there is $h\in H$ such that $h\cdot \Gamma \cap \Delta = \varnothing.$
\end{lemma}

\begin{lemma}\label{L:aris1.3}
For any   $\mathcal{L}$-structure $\mathcal{M}$,  the following statements are equivalent:
 \begin{enumerate}
     \item  $\mathcal{M}$  is highly algebraic;
     \item for all  $\bar{c}$ there exist $\bar{a}$, $b$ so that $\{\bar{a}\}, \{b\}, \{\bar{c}\}$ are pairwise disjoint and $b\in\mathrm{acl}(\bar{a}\bar{c},\mathcal{M})$; 
      \item for all  $\bar{c}$ there exist  $\bar{a}$, $b$ so that $\{\bar{a}\}, \{b\}, \{\bar{c}\}$ are pairwise disjoint and $b\in\mathrm{acl}(\bar{a},\mathcal{M})$; 
 \end{enumerate}
\end{lemma}
\begin{proof}
 It straightforward to see the equivalence between (1) and (2).  It is also clear that (3) implies (2). We prove that (2) implies (3) as well. Assume (2) holds and let $\bar{c}$ be given.

We may assume that $\mathrm{acl}(\emptyset,\mathcal{M})$ is finite since, otherwise, (3) would follow by simply setting  $\bar{a}:=\emptyset$. Let now $\bar{e}$ be a finite tuple enumerating $\mathrm{acl}(\emptyset,\mathcal{M})$. By (2), and by shrinking  $\bar{c}$ if necessary, we get $\bar{a}, b$ so that $\{\bar{a}\}, \{b\}, \{\bar{c}\}, \{\bar{e}\}$ and pairwise disjoint and $b\in\mathrm{acl}(\bar{a}\bar{c}\bar{e})$. 
\begin{claim}
There exists $g\in\mathrm{Aut}(\mathcal{M})_{\bar{e}}$ so that the tuple $g(\bar{a}b\bar{c})$ is disjoint from  $\bar{c}\bar{e}$.
\end{claim}
\begin{proof}[Proof of Claim]
Every $d\in\mathrm{dom}(\mathcal{M})\setminus \{\bar{e}\}$ has an infinite orbit  under  $\mathrm{Aut}(\mathcal{M})_{\bar{e}}$. Indeed, since $\bar{e}$ enumerates $\mathrm{acl}(\emptyset,\mathcal{M})$, there exist $h_1,\ldots,h_k \in \mathrm{Aut}(\mathcal{M})$ so that  every $g\in \mathrm{Aut}(\mathcal{M})$ is of the form $g=h_ih$ for some $h\in \mathrm{Aut}(\mathcal{M})_{\bar{e}}$ and $i\leq k$. Hence, if the set $E=\{gd \colon g\in  \mathrm{Aut}(\mathcal{M})_{\bar{e}} \}$ was finite, then so would be  the set $\{gd \colon g\in  \mathrm{Aut}(\mathcal{M})\}=\bigcup_i h_iE$.
Since  $\{\bar{a}b\bar{c}\}\subseteq\mathrm{dom}(\mathcal{M})\setminus \{\bar{e}\}$, by Lemma \ref{Thm:Neu} there exists a $g$ as in the statement of the claim.
\end{proof}

Fix $g$ as in the claim and set $\bar{a}':=g(\bar{a}\bar{c})$ and $b':=g(b)$. By the claim, and since $\{\bar{a}\}, \{b\}$  and $\{\bar{c}\}$  were pairwise disjoint, so are $\{\bar{a}'\}, \{b'\}$, and  $\{\bar{c}\}$. But then we have that 
\[b\in\mathrm{acl}(\bar{a}\bar{c}\bar{e},\mathcal{N}) \implies 
g(b)\in\mathrm{acl}(g(\bar{a}\bar{c})g(\bar{e}),\mathcal{N}) \implies 
b'\in\mathrm{acl}(\bar{a}'\bar{e},\mathcal{N})\implies b'\in\mathrm{acl}(\bar{a}',\mathcal{N}), \]
where the last implication follows from the fact that 
$\bar{e}\in \mathrm{acl}(\emptyset,\mathcal{N})$, by an argument similar to the one in the previous claim.
\end{proof}

Using the above reformulation (3) of  high algebraicity, we get the following dichotomy. See Examples \ref{Ex:1}, \ref{Ex:2} for simple examples of structures on each side of the dichotomy.

\begin{lemma}\label{L:aris1.5}
If $\mathcal{N}$ is highly algebraic, then for every finite tuple $\bar{c}$ there exist  $\bar{a}$ and  $b$ in $\mathrm{dom}(\mathcal{N})$ so that the sets $\{\bar{a}\}$, $\{b\}$, and $\{\bar{c}\}$ are pairwise disjoint and one of following holds:
\begin{enumerate}
\item $\{\bar{a}\}=\{a\}$ is a singleton and both $b\in \mathrm{acl}(a,\mathcal{N})$ and $a\in \mathrm{acl}(b,\mathcal{N})$ hold; or
\item $b\in \mathrm{acl}(\bar{a},\mathcal{N})$ and there exists and infinite sequence $(\bar{a}^i\colon i\in \mathbb{N})$ of  pairwise disjoint tuples so that  for every $i\in\mathbb{N}$ we have that $\mathrm{qft}(\bar{a}^ib,\mathcal{N})=\mathrm{qft}(\bar{a}b,\mathcal{N})$.
\end{enumerate}
\end{lemma}
\begin{proof}

By Lemma \ref{L:aris1.3}(3) there exist
 $\bar{a}$ and  $b$ in $\mathrm{dom}(\mathcal{N})$ so that the sets $\{\bar{a}\}$, $\{b\}$, and $\{\bar{c}\}$ are pairwise disjoint and  $b\in \mathrm{acl}(\bar{a},\mathcal{N})$. Assume that there is no such  $\bar{a}$ and  $b$ satisfies (2) above.   Choose   $\bar{a}$ and $b$ as in Lemma \ref{L:aris1.3}(3) so that  $\bar{a}$ is of minimum possible size. Since by assumption  $\bar{a}$,  $b$ do not satisfy (2) we have that $\bar{a}\neq \emptyset$.
 
 Let now $a$ be any entry of $\bar{a}$ so that $a\in \mathrm{acl}(b,\mathcal{N})$. Such  entry exists since, otherwise, by Lemma \ref{Thm:Neu} we would have that (2) holds for this $\bar{a}$ and  $b$.  But since $\{b\}$ is a singleton for which there is $a\in \mathrm{acl}(b,\mathcal{N})$, the minimality assumption on the length of $\bar{a}$ above implies that  $\{\bar{a}\}=\{a\}$ and hence  (1) above follows. 
\end{proof}

\subsection{Proof  of Lemma \ref{L:aris2}} \label{SS:2} We may now conclude the proof  of Lemma \ref{L:aris2} which, in turn, establishes  the $(\Rightarrow)$ direction of Theorem \ref{T:main}. Recall that throughout this section we assume that the structure $\mathcal{N}\in\mathrm{Str}(\mathcal{L},\mathbb{N})$  is  highly algebraic and  $R$--ultrahomogeneous.

\begin{proof}[Proof of Lemma \ref{L:aris2}]
Let $K\subseteq X$ be compact. 

For every $s\in 2^{<\mathbb{N}}$ we will define a natural number $\ell_s$ and some permutation $\gamma_s$ of the finite set $\ell_s=\{0,1,\ldots,\ell_s-1\}$ so that for all $s\in 2^{<\mathbb{N}}$ we have that:
\begin{enumerate}
    \item[(i)]  $\ell_{\emptyset}=0$ and $\gamma_{\emptyset}$ is the empty permutation;
    \item[(ii)] $\ell_{s^{\frown}0}= \ell_{s^{\frown}1}$  and 
    $\ell_{s^{\frown}0},\ell_{s^{\frown}1}> \ell_s$, as well as $\gamma_{s^{\frown}0}\upharpoonright \ell_s=\gamma_s$ and $\gamma_{s^{\frown}1}\upharpoonright \ell_s=\gamma_s$;
    \item[(iii)] $g_0K\cap g_1 K=\emptyset$, for all $g_0,g_1\in \mathrm{Sym}(\mathbb{N})$ with $g_0\upharpoonright \ell_{s^{\frown}0}=\gamma_{s^{\frown}0}$ and $g_1\upharpoonright \ell_{s^{\frown}1}=\gamma_{s^{\frown}1}$.
\end{enumerate}

Granted $(\gamma_s\colon s\in 2^{<\mathbb{N}})$ as above, we define the desired family $(g_{\alpha}\colon \alpha\in 2^{\mathbb{N}})$ by setting
\[g_{\alpha}:=\bigcup_{n\in\mathbb{N}} \gamma_{\alpha|n}\]
Then, (ii) implies that $g_{\alpha}\in \mathrm{Sym}(\mathbb{N})$ and (iii) implies that $g_{\alpha} K \cap g_{\beta} K =\emptyset$, whenever $\alpha\neq \beta$.

We proceed now to the definition of $\ell_s,\gamma_s$ by induction on $s$. Assume that $\ell_s, \gamma_s$ have been defined. We define $\ell_{s^{\frown}0},\ell_{s^{\frown}1}$ and $ \gamma_{s^{\frown}0}, \gamma_{s^{\frown}1}$ using the following construction.

\begin{claim}\label{Claim:1in}
There exists a finite family  $\{S_i\colon i\in I\}$ of  $\ell_s$-ary relations $S_i$ so that:
\begin{enumerate}
    \item[(I)] For every $\mathcal{M}\in K$ there is $i\in I$ so that $\mathcal{M}\models S_i(0,1,\ldots,\ell_s-1)$;
    \item[(II)] For every $\ell_s$-ary tuple $\bar{e}$ with $\mathcal{N}\models S_i(\bar{e})$ we have that $S_i$ captures  $\bar{e}$ in $\mathcal{N}$.
\end{enumerate}
\end{claim}
\begin{proof}[Proof of Claim]
 This follows from compactness of $K$ since $\mathcal{N}$ is $R$--ultrahomogeneous.
\end{proof}

For each $i\in I$ let $\bar{c}^i$ be a tuple in $\mathbb{N}$ so that $\mathcal{N}\models S_i(\bar{c}^i)$ and let $\bar{c}$ be the union of all these tuples. Apply Lemma \ref{L:aris1.5} for this $\bar{c}$ to get $\bar{a}$,  $b$ so that $\{\bar{a}\}, \{b\}, \{\bar{c}\}$ are pairwise disjoint and for which either (1) or (2) of Lemma \ref{L:aris1.5} holds.  Since $\mathcal{N}$ is $R$--ultrahomogeneous, we may  pick  relations $P,Q,R\in \mathcal{L}$, so that: $P$ captures $\mathrm{qft}(\bar{a},\mathcal{N})$; $Q$ captures $\mathrm{qft}(b,\mathcal{N})$; and $R$ captures $\mathrm{qft}(\bar{a}b,\mathcal{N})$.
Below we will abuse notation and use the letters 
$\bar{a}$ and $b$ for potentially new tuples which satisfy  $P$ and $Q$, respectively, in any structure $\mathcal{M}$ with $\mathcal{M}\simeq_{\mathrm{iso}}\mathcal{N}$.

We define $\ell_{s^{\frown}0}, \ell_{s^{\frown}1},\gamma_{s^{\frown}0},\gamma_{s^{\frown}1}$ by considering separately cases (1) \& (2) of Lemma \ref{L:aris1.5}.

\medskip{}

\noindent {\bf Case (1).} Assume that Lemma \ref{L:aris1.5}(1) holds for $\bar{a}, b$,  $\bar{c}$. In particular, we have that both $P,Q$ are unary  relations and $R$ is a binary relation.

\begin{claim}
There exists $\ell>\ell_s$ and some set $Z\subseteq \mathbb{N}\setminus \{0,\ldots,\ell-1\}$ with $|Z|=\ell-\ell_s$, so that for every $\mathcal{M}\in K$ we have that:
\begin{enumerate}
    \item[(a)] there exist $a,b\in \{\ell_s,\ldots,\ell-1\}$ so that $\mathcal{M}\models R(a,b)$;
    \item[(b)] for every two $p,q\in Z$  we have that  $\mathcal{M}\not\models R(p,q)$.
\end{enumerate}
\end{claim}
\begin{proof}[Proof of Claim]
By the choice of $R$ and $S_i$, for every $i\in I$ we have that 
\[\mathcal{N}\models \forall \bar{z}  \;\big(  S_i(\bar{z})  \implies \exists x \not \in \bar{z}  \;
\;\exists y\not \in \bar{z}  \; \;  R(x,y)\big)\]
Hence, by Claim \ref{Claim:1in}, for every $\mathcal{M}\simeq_{\mathrm{iso}}\mathcal{N}$ there exists a least number $n(\mathcal{M})\geq \ell_s$ so that  $\mathcal{M}\models R(a,b)$ for some  $a,b$ with $\ell_s\leq a,b <n(\mathcal{M})$. Since $\mathcal{M}\mapsto n(\mathcal{M})$ is a continuous map, by compactness of $K$ we can find $\ell:=\sup_K n(\mathcal{M})$ satisfying property (a) above.

We next build $Z$ inductively,  adding one point at a time, as follows. Assume that some fragment $Y$ of $Z$ has been defined so that for all
$\mathcal{M}\in K$ property (b) above holds for $Y$ in place of $Z$.
By compactness of $K$ and the fact that for each  $\mathcal{M}\simeq_{\mathrm{iso}}\mathcal{N}$, and every fixed $p,q\in\mathbb{N}$, there exist at most finitely many solutions to the formulas $\varphi(y) \equiv R(p,y)$ and $\psi(x) \equiv R(x,q)$,  we may find  $\ell'> \max(Y), \ell$, so that for every  $\mathcal{M}\in K$ we have that:
\begin{itemize}
    \item[(i)]    for every $p\in Y$ and every $q\geq \ell'$ we have that  $\mathcal{M}\not\models R(p,q)$;
    \item[(ii)]   for every $q\in Y$ and every $p\geq \ell'$ we have that  $\mathcal{M}\not\models R(p,q)$.
\end{itemize}
It follows that  property (b) above holds for $Y':=Y\cup\{\ell'\}$ in place of $Z$. Since $|Y'|>|Y|$, by repeating this process enough times we define $Z$, with $|Z|=\ell\setminus \ell_s$, satisfying  (b) above.
\end{proof}

Let now $\ell$ and $Z$ as in the last claim and set $\ell_{s^{\frown}0}, \ell_{s^{\frown}1}:=\max(Z)+1$. Let $\gamma_{s^{\frown}0}$ be the permutation on the set $\ell_{s^{\frown}0}$, with  $\gamma_{s^{\frown}0}\upharpoonright \ell_s= \gamma_{s}$, which is identity on $\ell_{s^{\frown}0}\setminus\ell_s$.
Let also $\gamma_{s^{\frown}1}$ be any permutation on the set $\ell_{s^{\frown}1}$, with  $\gamma_{s^{\frown}1}\upharpoonright \ell_s= \gamma_{s}$, which induces a bijection between the sets  $\ell\setminus \ell_s$ and $Z$.
 We are left to show that  (iii) above holds.

Indeed, let  $g_0,g_1\in \mathrm{Sym}(\mathbb{N})$ with $g_0\upharpoonright \ell_{s^{\frown}0}=\gamma_{s^{\frown}0}$ and $g_1\upharpoonright \ell_{s^{\frown}1}=\gamma_{s^{\frown}1}$ and assume towards contradiction that $g_0K\cap g_1 K\neq\emptyset$. Set $g:=g^{-1}_1g_0$ and fix some $\mathcal{M}\in K$ so that $g_*\mathcal{M}\in K$. Since $\mathcal{M}\in K$, by (a) above we get $a,b<\ell$  so that $\mathcal{M}\models R(a,b)$. Hence,  $g_*\mathcal{M}\models R(ga,gb)$.
But since $g_*\mathcal{M}\in K$,  by (b) above and the fact that  $ga, gb \in Z$ we have $g_*\mathcal{M}\not\models R(ga,gb)$; a contradiction. If follows that $g_0K\cap g_1 K=\emptyset$ as desired.

\bigskip{}

\noindent {\bf Case (2).} Assume that Lemma \ref{L:aris1.5}(2) holds for $\bar{a}, b$, and $\bar{c}$. 

\begin{claim}
There exist $\ell,\ell',\ell''\in\mathbb{N}$ with $\ell''>\ell'>\ell>\ell_s$ so that for all $\mathcal{M}\in K$ we have:
\begin{enumerate}
    \item[(a)] there exists $b\in\mathbb{N}$, with $\ell_s \leq b<\ell$, so that $\mathcal{M}\models Q(\bar{b})$.
    \item[(b)] for all $b<\ell$ with  $\mathcal{M}\models Q(\bar{b})$, there is $\bar{a}$ in $\ell'\setminus \ell$ so that $\mathcal{M}\models R(\bar{a},b)$.
   \item[(c)] for all $\bar{a}$ in $\ell'\setminus \ell$ with  $\mathcal{M}\models P(\bar{a})$, there exists no $b\geq \ell''$ so that $\mathcal{M}\models R(\bar{a},b)$. 
\end{enumerate}
\end{claim}
\begin{proof}[Proof of Claim]

By compactness of $K$ we can find some $\ell>\ell_s$ so that for all $\mathcal{M}\in K$ property (a) holds. Since each $\mathcal{M}$ with $\mathcal{M}\simeq_{\mathrm{iso}}\mathcal{N}$ is $R$--ultrahomogeneous and $P,Q,R$ capture the pertinent quantifier free types, for any $\mathcal{M}\in K$ and any fixed  $b\in\mathbb{N}$ with  $\mathcal{M}\models Q(b)$ there exists   $\bar{a}$ so that $\mathcal{M}\models R(\bar{a},b)$. In fact, since we are under the assumption that Lemma \ref{L:aris1.5}(2) holds, for  every $\mathcal{M}\in K$ we can can always find such $\bar{a}$ outside of $\{0,1,\ldots,\ell-1\}$.  Hence by a second application of compactness of $K$  we may find some $\ell'>\ell$ so that property (b) holds as well. 
By a third and final application of compactness
of $K$, and since for each  $\mathcal{M}\in K$ and every $\bar{a}\in\mathbb{N}$ there are only finitely many solutions to the formula $\varphi(y) \equiv R(\bar{a},y)\wedge P(\bar{a})$, we may find  $\ell''>\ell'$ for which  property (c) holds.
\end{proof}

Let now $\ell,\ell',\ell''$ as in the above claim and set $\ell_{s^{\frown}0}, \ell_{s^{\frown}1}:=\ell''+\ell$.

Let $\gamma_{s^{\frown}0}$ be the permutation on the set $\ell_{s^{\frown}0}$, with  $\gamma_{s^{\frown}0}\upharpoonright \ell_s= \gamma_{s}$, which is identity on $\ell_{s^{\frown}0}\setminus\ell_s$.
Let also $\gamma_{s^{\frown}1}$ be any permutation on the set $\ell_{s^{\frown}1}$, with  $\gamma_{s^{\frown}1}\upharpoonright \ell_s= \gamma_{s}$,  which  fixes every $k\in\{\ell,\ldots,\ell''-1\}$, but which exchanges $k$ and  $k+\ell''$, for all  $k$ with $\ell_s\leq k <\ell$. We are left to show that  (iii) above holds.

Let  $g_0,g_1\in \mathrm{Sym}(\mathbb{N})$ with $g_0\upharpoonright \ell_{s^{\frown}0}=\gamma_{s^{\frown}0}$ and $g_1\upharpoonright \ell_{s^{\frown}1}=\gamma_{s^{\frown}1}$ and assume towards contradiction that $g_0K\cap g_1 K\neq\emptyset$. Set $g:=g^{-1}_1g_0$ and fix some $\mathcal{M}\in K$ so that $g_*\mathcal{M}\in K$. Since $\mathcal{M}\in K$, by (a) and (b) above we get $b$ in $\ell\setminus \ell_s$ and $\bar{a}$ in $\ell'\setminus \ell$ so that $\mathcal{M}\models R(\bar{a},b)$. Hence,  $g_*\mathcal{M}\models R(g\bar{a},gb)$, and since $g$ fixes $\bar{a}$ we have that $g_*\mathcal{M}\models R(\bar{a},gb)$. But then, by applying (c) to $g_*\mathcal{M}\in K$,
we get a contradiction with the fact that $gb\geq \ell''$.
\end{proof}

\section{Structures which are not highly algebraic are quasi--random}

The following is the main result of this section.

\begin{theorem}\label{T:main2N} Let $\mathcal{M}$ be a $\mathcal{L}$-structure such that there is finite $B\subset \N$ satisfying for all $\bar{a}\in \N$, $\mathrm{acl}(\bar{a},\mathcal{M})\subseteq \{\bar{a}\}\cup B$. Then there is a  $\mathrm{Sym}(\mathbb{N})$-quasi--invariant probability measure on $\mathrm{Str}(\mathcal{L},\N)$ concentrated on structures isomorphic to $\mathcal{M}$. 
\end{theorem}
\begin{proof}

We may assume without loss of generality that  $\mathcal{L}$ only contains relation symbols. Indeed,  if $\mathcal{L}$ contained function symbols, then we can always move a
the language $\Ell'$ which contains all relational symbol of $\Ell$; and for each $f\in \Ell$ function symbol of arity $r$, $\Ell'$ contains a relation symbol $R_f$ of arity $r+1$. Similarly, from $\mathcal{M}$, move to the $\Ell'$-structure $\mathcal{M}'$, with: $\mathrm{dom}(\mathcal{M})=\N=\mathrm{dom}(\mathcal{M}')$;  
relations from $\mathcal{L}$ are interpreted as in $\mathcal{M}$; and for each function symbol $f$ from $\mathcal{L}$, set $R_f^{\mathcal{M}'}(\bar x, a) \Leftrightarrow f^{\mathcal{M}}(\bar x) =a$.  
Notice that $\mathcal{M}$ and $\mathcal{M}'$ have  the same automorphism group (as permutation groups). In particular, $\mathcal{M}'$ is highly algebraic iff $\mathcal{M}$ is highly algebraic. Moreover this construction induces an $S_{\infty}$-equivariant  Borel isomorphism between  $[\mathcal{M}]_{\mathrm{iso}}\subseteq\mathrm{Str}(\mathcal{L},\N)$ and $[\mathcal{M}']_{\mathrm{iso}}\subseteq\mathrm{Str}(\mathcal{L}',\N)$, the existence of a quasi-invariant probability measure on $[\mathcal{M}']_{\mathrm{iso}}$ implies the same for $[\mathcal{M}]_{\mathrm{iso}}$.

Let now  $B$  as in the statement of Theorem \ref{T:main2N}. We view the set $B$ as enumerated  by $\bar{b}=(b_0,\ldots,b_{l-1})\in \N^{\ell}$, with $|B|=\ell$, and we fix some new ``atom" $z\notin \N$.  

For each $\bar{c}\in \mathbb{N}^{\ell}$  we  consider a new language $\mathcal{L}_{\bar{c}}$. This language  contains, for each  $R\in \mathcal{L}$ of arity $n$ and each  $f\colon \{0,\ldots,n-1\}\to \{\bar{c}\}\cup \{z\}$,   a symbol $R^{\bar{c}}_f$ of arity $|f^{-1}(z)|$. In addition, given any tuple $\bar{a}\in (\N\setminus \{\bar{c}\})^k$ and any map $f\colon \{0,\ldots,n-1\}\to \{\bar{c}\}\cup \{z\}$ such that $|f^{-1}(z)|=k$, we define $\bar{a}_f\in \N^{n}$, as the unique tuple such that:
\begin{itemize}
    \item[$i)$]  for each $i\in n$, with $f(i)\in \{\bar{c}\}$,  the $i$-th coordinate of $\bar{a}_f$ is $f(i)$,
    \item[$ii)$] removing all entries of $\bar{a}_f$ that are in $\{\bar{c}\}$ gives $\bar{a}$.
\end{itemize}

For any $\mathcal{N}\in  \mathrm{Str}(\mathcal{L},\N)$ and any fixed $\bar{c}\in \mathbb{N}^{\ell}$, we define $\mathcal{N}_{\bar{c}}\in \mathrm{Str}(\mathcal{L}_{\bar{c}},\N\setminus \{\bar{c}\})$  by setting
\[\mathcal{N}_{\bar{c}}\models R^{\bar{c}}_f(\bar{a}) \iff  \mathcal{N}\models R(\bar{a}_f),\]
for every  $R^{\bar{c}}_f\in \mathcal{L}_{\bar{c}}$. Notice that there is a natural isomorphism of the permutation groups   $\mathrm{Aut}(\mathcal{N}_{\bar{c}})$  and  $\mathrm{Aut}(\mathcal{N})_{\bar{c}}$, defined by extending each $g\in\mathrm{Aut}(\mathcal{N}_{\bar{c}})$ to the identity on $\{\bar{c}\}$.

\begin{claim}\label{Claim:NoAlg}
The $\mathcal{L}_{\bar{b}}$-structure   $\mathcal{M}_{\bar{b}}$ has no algebraicity.
\end{claim}
\begin{proof}[Proof of Claim]
Assume that  $d\in \N\setminus \{\bar{b}\}$ has finite orbit under $\mathrm{Aut}(\mathcal{M}_{\bar{b}})_{\bar{a}}$. Then $d$ also has finite orbit under $\mathrm{Aut}(\mathcal{M})_{\bar{a}\bar{b}}$, since  $\mathrm{Aut}(\mathcal{M}_{\bar{b}})_{\bar{a}}$ and  $\mathrm{Aut}(\mathcal{M})_{\bar{a}\bar{b}}$ are the same permutation groups on $\N \setminus \{\bar{b}\}$. Therefore $d$ must be in $\{\bar{a}\}\cup \{\bar{b}\}$. Since  $d\not\in \{\bar{b}\}$, we have $d\in\{\bar{a}\}$.
\end{proof}

Let now $W$ be the collection of all pairs $(\bar{c},\mathcal{N})$,  where $\bar{c}\in \N^{\ell}$ and $\mathcal{N}\in \mathrm{Str}(\mathcal{L}_{\bar{c}},\N \setminus \{\bar{c}\})$.
 We define an analog of the logic action of $\mathrm{Sym}(\N)$ on $W$ as follows: for $g\in \mathrm{Sym}(\N)$ and $(\bar{c},\mathcal{N})\in W$  we set $g\cdot(\bar{c},\mathcal{N}):=(g(\bar{c}),g_*\mathcal{N})$, where  $g_*\mathcal{N}\in \mathrm{Str}(\mathcal{L}_{g(\bar{c})},\N\setminus\{ g(\bar{c})\})$ is defined by naturally extending the formulas from Section \ref{SS:Background:2}. That is, by 
 setting, for all $R^{\bar{c}}_f\in \mathcal{L}_{\bar{c}}$,
 \[g_*\mathcal{N}\models R^{g(\bar{c})}_{g\circ f}(\bar{a})\iff \mathcal{N}\models R^{\bar{c}}_f(g^{-1}(\bar{a}))\]

Let   $Y$  be the $\mathrm{Sym}(\N)$-orbit of $(\bar{b},\mathcal{M}_{\bar{b}})$ in $W$ and consider the map $\beta \colon Y \to  \mathrm{Str}(\mathcal{L},\N)$ which is defined as follows:  for all $(\bar{c},\mathcal{N})\in Y$ and $R\in\mathcal{L}$, ${\beta(\bar{c},\mathcal{N})}\models R(\bar{a})$ holds if and only if ${\mathcal{N}}\models R^{\bar{c}}_{f}(\bar{e})$ holds, for the unique   $f\colon n \to \{\bar{c}\}\cup \{z\}$ and  $\bar{e}\in (\N \setminus \{\bar{c}\})^k$ such that $\bar{e}_f=\bar{a}$.

\begin{claim}
    $\beta$ is  $S_\infty$-equivariant.
\end{claim}

\begin{proof}[Proof of Claim]
Let $g\in \mathrm{Sym}(\N)$ and $(\bar{c},\mathcal{N})\in Y$, we need to prove that
\[\beta(g (\bar{c}), g_{*}\mathcal{N})= g_* \beta(\bar{c},\mathcal{N}).\]  
To see this, let $R\in \mathcal L$ and  $\bar{a}$ in $\mathbb{N}$. We have:
\begin{align*}
      \beta(g (\bar{c}), g_{*}\mathcal{N}) \models R(\bar{a}) \quad \Leftrightarrow \quad &      {g_{*}\mathcal{N}} \models R^{g(\bar{c})}_{f}(\bar{e}), \text{for the unique }f\colon n\to \{g(\bar{c})\}\cup \{z\}\\ & \text{ and  }
      \bar{e}\in (\N \setminus \{g  (\bar{c})\})^k\text{ such that }\bar{e}_f=\bar{a}.\\
 \quad \Leftrightarrow \quad &      {\mathcal{N}} \models R^{\bar{c}}_{g^{-1}\circ f}(g^{-1}(\bar{e})), \text{for the unique }f\colon n\to \{g(\bar{c})\}\cup \{z\}\\ & \text{ and  }
      \bar{e}\in (\N \setminus \{g  (\bar{c})\})^k\text{ such that }\bar{e}_f=\bar{a}.\\
 \quad \Leftrightarrow \quad &      {\mathcal{N}}\models R^{\bar{c}}_{h}(\bar{d}), \text{for the unique }h\colon n\to \{\bar{c}\}\cup \{z\}\\ & \text{ and  }
      \bar{d}\in (\N \setminus \{\bar{c}\})^k\text{ such that }g(\bar{d}_h)=\bar{a}, \\
      & \text{ namely,  for } h:=g^{-1}\circ f \text{ and } \bar{d}:=g^{-1}(\bar{e}). \\ 
 \quad \Leftrightarrow \quad &      {\mathcal{N}}\models R^{\bar{c}}_{f}(\bar{e}), \text{for the unique } f \colon n\to \{\bar{c}\}\cup \{z\}\\ & \text{ and  }
      \bar{e}\in (\N \setminus \{\bar{c}\})^k\text{ such that } \bar{e}_f=g^{-1}(\bar{a}), \\
   \quad \Leftrightarrow \quad &  \beta(\bar{c},\mathcal{N}) \models R(g^{-1}(\bar{a}))\\
  \quad \Leftrightarrow \quad &   {g_* \beta(\bar{c},\mathcal{N})}\models R(\bar{a}).
\end{align*}
\end{proof}

\begin{claim}\label{Claim:Isom}
  $\beta(Y)= [\mathcal{M}]_{\mathrm{iso}}$
\end{claim}
\begin{proof}
This follows from the previous claim, since $\mathcal{M}=\beta(\mathcal{M}_{\bar{b}},\bar{b})$.
\end{proof}

\begin{claim}\label{Claim:Borel}
 $Y$ is a standard Borel space and  $\beta$ is a Borel map. \color{black}
\end{claim}
\begin{proof}
Notice first that $W$ is a Polish space as it is the countable disjoint union of the Polish spaces $\mathrm{Str}(\mathcal{L}_{\bar{c}},\N\setminus \{\bar{c}\})$, where the union ranges over all $\bar{c}\in \N^\ell$. Since $Y$ is a single orbit of the continuous action of $\mathrm{Sym}(\N)$ on $W$ we have that $Y$ is a Borel subset of $W$, and hence, a standard Borel space. It is straightforward to check that the map $\beta$ is, in fact, continuous, as whether ${\beta(\bar{c},\mathcal{N})}\models R(\bar{a})$ holds depends only on whether  ${\mathcal{N}}\models R^{\bar{c}}_{f}(\bar{e})$ holds, for a single relation  $R^{\bar{c}}_{f}$ and a single tuple $\bar{e}$.
\end{proof}

\begin{claim}
    There is a measure $\lambda$ on $Y$ that is $\mathrm{Sym}(\N)$-quasi--invariant.
\end{claim}

\begin{proof}
    First, using  \cite{AFP}*{Theorem 1.1} and Claim \ref{Claim:NoAlg}, we may fix some  $\mathrm{Sym}(\N\setminus \{\bar{b}\})$-invariant measure  $\mu_{\bar{b}}$ on $\mathrm{Str}(\mathcal{L}_{\bar b},\N \setminus \{\bar{b}\})$, which concentrates on the orbit of $\mathcal{M}_{\bar{b}}$. We now define a measure on $\mathrm{Str}(\mathcal{L}_{\bar{c}},\N \setminus \{\bar{c}\})$ by taking any element $g$ of $\mathrm{Sym}(\N)$ sending $\bar b$ to $\bar c$ and setting $\mu_{\bar{c}}=g_*\mu_{\bar b}$. The obtained measure is $\mathrm{Sym}(\N\setminus \{\bar{c}\})$-invariant and does not depend on the choice of $g$ by $\mathrm{Sym}(\N\setminus \{\bar{b}\} )$-invariance of $\mu_{\bar b}$. Denote by $\delta_{\bar c}$ the measure on $\N^\ell$ such that for $A\subseteq \N^\ell$ , $\delta_{\bar{c}}(A)=1$ if $\bar{c}\in A$  and $0$ otherwise. Now take any fully supported measure $\nu$ on the countable set $\N^\ell$ and integrate $\delta_{\bar c} \otimes \mu_{\bar c}$ along $\nu$, since $\N^{\ell}$ is countable this is always well-defined. We denote by $\lambda $ the obtained measure on $Y$, as for every $\bar c$, $\delta_{\bar c} \otimes \mu_{\bar c}$ is a measure on $Y$.

To see that $\lambda$ is  $\mathrm{Sym}(\N)$-quasi--invariant, let $A\subseteq Y$ with $\lambda(A)>0$, and pick  $\bar{d}\in \N^\ell$ with
\[\lambda\big( \{(\bar{c},\mathcal{N})\in A \colon \bar{c}=\bar{d} \} \big)>0. \]
This implies that 
\[\mu_{\bar{d}}\big(\{\mathcal{N} \in \mathrm{Str}(\mathcal{L}_{\bar d},\N \setminus \{\bar{d}\}) \colon (\bar{d},\mathcal{N})\in A \} \big)>0.\]
Let now $g\in \mathrm{Sym}(\N)$ and notice that, by the definition $\mu_{\bar{c}}$, we have $\mu_{g(\bar{d})}=g_*\mu_{\bar{d}}$. Hence,
\[\mu_{g(\bar{d})}\big(\{g_*\mathcal{N} \in \mathrm{Str}(\mathcal{L}_{g(\bar d)},\N \setminus \{g(\bar{d})\}) \colon (\bar{d},\mathcal{N})\in A \} \big)>0.\]
Since $\nu$ is fully supported, we have that $\lambda(g(A))\geq \lambda\big( \{(\bar{c},\mathcal{N})\in g(A) \colon \bar{c}=g(\bar{d}) \} \big)>0$.

\end{proof}

To conclude with the proof of Theorem \ref{T:main2N}, we simply push forward  $\lambda$ via $\beta$. By  Claim \ref{Claim:Isom} and Claim \ref{Claim:Borel} we have an $\mathrm{Sym}(\mathbb{N})$-quasi--invariant measure on $[\mathcal{M}]_{\mathrm{iso}}$ which can be normalized to a probability measure without losing $\mathrm{Sym}(\mathbb{N})$-quasi--invariance.
\end{proof}

\section{Examples}

We start by providing two examples of simple, highly algebraic structures. Each  example  realizes the respective side of the dichotomy that we established in Lemma \ref{L:aris1.5}.

\begin{example}\label{Ex:1}
Let $\mathcal{L}:=\{R\}$, where $R$ is a binary relation, and let $\mathcal{N}$ be a graph that comprises of the disjoint union of countably infinite many edges. For example, set $\mathcal{N}=(\mathbb{N},R^{\mathcal{N}})$, where for all $n,m\in\mathbb{N}$ we have that $(n,m)\in R^{\mathcal{N}}$ if and only if there is  $k\in\mathbb{N}$ so that $\{n,m\}=\{2k,2k+1\}$. Clearly $\mathcal{N}$ is highly algebraic and satisfies  (1) of Lemma \ref{L:aris1.5}. 
\end{example}

\begin{example}\label{Ex:2}
Let $\mathcal{L}:=\{R\}$, where $R$ is a binary relation, and let $\mathcal{N}$ be a graph that comprises of the disjoint union of countably many copies of a bipartite  graph whose one part is countably infinite and the other is a singleton.  For example, set $\mathcal{N}=(\mathbb{N}^2,R^{\mathcal{N}})$, where $\big((n_0,n_1),(m_0,m_1)\big)\in R^{\mathcal{N}}$ if and only if   $m_0=n_0$ and either of the following two hold: $n_1=0$ and $m_1>0$; or $m_1=0$ and $n_1>0$.  Clearly $\mathcal{N}$ is highly algebraic and (its $R$-homogeneous expansion from Lemma \ref{L:Ultrahomogeneous}) satisfies  (2) of Lemma \ref{L:aris1.5}. 
\end{example}

We close with some examples of structures which have algebraicity but they are not highly algebraic. In view of  \cite{AFP} and Theorem \ref{T:main}, these are quasi-random structures which are not random.

\begin{example}[Finite Inclusions]
Let $\mathcal{N}$ be any ultrahomogeneous structure without algebraicity and let $F$ be a finite substructure of it. Then the structure $(\mathcal{N},U_F)$, which endows $\mathcal{N}$  with a unary predicate interpreted on $F$, is not highly-algebraic and hence, by Theorem \ref{T:main}, quasi-random. 
\end{example}

\begin{example}[Generic Finite Quotients]
Let $\mathcal{N}$ be any ultrahomogeneous structure without algebraicity and let $\pi\colon \mathrm{dom}(N)\to B$ be a generic function with $B$ being a  finite set. Consider the structure $\mathcal{N}_{B}$ given by setting   $\mathrm{dom}(\mathcal{N}_{B}):=\mathrm{dom}(\mathcal{N})\cup B$ and adding  the map $\pi$, on top of the existing on $\mathcal{N}$,  as the interpretation of a new function symbol. It is not difficult to check that  $\mathcal{N}_B$ is not highly-algebraic and hence, by Theorem \ref{T:main}, quasi-random. 
\end{example}

\end{document}